\documentclass{amsart}
\usepackage[utf8]{inputenc}
  \usepackage{amsmath}
\usepackage{enumerate}
  \usepackage{amssymb}
  \usepackage{amsthm}
  \usepackage{appendix}
  \usepackage{bm}
  \usepackage{graphicx}
  \usepackage{color}
\usepackage{epsfig}
\usepackage{amsmath,amssymb}
\hoffset - 0.5 in
\newtheorem{Th}{Theorem}[section]
\newtheorem{lem}[Th]{Lemma}

\theoremstyle{definition}

\newtheorem{Prop}[Th]{Proposition}

\theoremstyle{remark}
\newtheorem*{rem}{\bf Remark}
\newtheorem*{que}{\bf Questions}
\newtheorem*{Conj}{\bf Conjecture}
\newtheorem*{thank}{\ \ \ \bf Acknowledgment}

\numberwithin{equation}{section}

\newcommand{\tend}[3][]{\xrightarrow[#2\to#3]{#1}}

\newcommand{\egdef}{\stackrel{\textrm {def}}{=}}
\newcommand{\ds}{\displaystyle}

\newcommand{\Z}{\mathbb{Z}}

\newcommand{\N}{\mathbb{N}}
\newcommand{\T}{\mathbb{T}}

\newcommand*{\QEDA}{\hfill\ensuremath{\blacksquare}}
\title[The Rudin-Shapiro polynomials \& the Fekete polynomials]{The Rudin-Shapiro polynomials and The Fekete polynomials are not $L^\alpha$-flat}

\author{el Houcein el Abdalaoui}
\address{Normandie University, University of Rouen
  Department of Mathematics, LMRS  UMR 60 85 CNRS\\
Avenue de l'Universit\'e, BP.12
76801 Saint Etienne du Rouvray - France .}
\email{elhoucein.elabdalaoui@univ-rouen.fr}

\urladdr{http://www.univ-rouen.fr/LMRS/Persopage/Elabdalaoui/}

\date{\today}
\dedicatory{}

\keywords{Simple Lebesgue spectrum, Banach's problem, Rokhlin's problem, weak Rokhlin's problem,  merit factor, flat polynomials, ultraflat polynomials, Littlewood's problem, Golay-Rudin-Shapiro sequence, Rudin-Shapiro polynomials, Fekete polynomials, Lipschitz functions, Fourier series, substitution, Barker sequence, digital transmission.
$$
\includegraphics[scale=0.15]{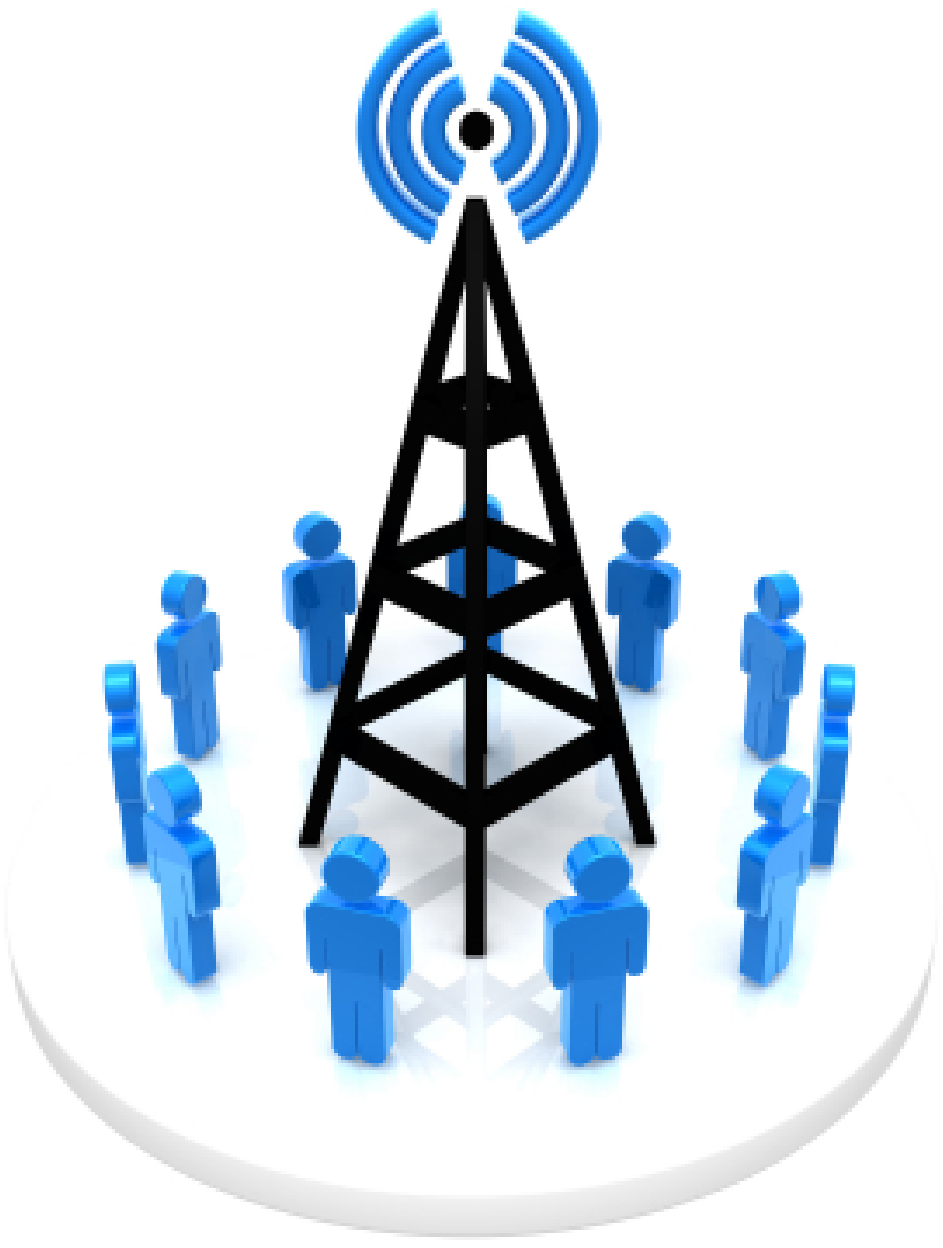}
$$
}
\begin{document}
\maketitle
\begin{abstract} We establish that the Rudin-Shapiro polynomials are not $L^\alpha$-flat, for any $\alpha \geq 0$. We further prove that the ``truncated'' Rudin-Shapiro sequence cannot generate a sequence of $L^\alpha$-flat polynomials, for any $\alpha \geq 0$. In the appendix, we present a simple proof of the fact that the Fekete polynomials and the modified or shifted Fekete polynomials  are not $L^\alpha$-flat, for any $\alpha \geq 0$.

\end{abstract}

\section{Introduction}\label{intro}
  The purpose of this note is to establish that the Rudin-Shapiro polynomials are not $L^\alpha$-flat, for any $\alpha \geq 0$. This answers a question raised by B. Weiss \cite{Weiss}. We thus obtain that the equivalence of the $L^\alpha$-norms does not imply $L^\alpha$-flatness. We further provide a simple and direct proof that the Fekete polynomials and the modified or shifted Fekete polynomials are not $L^\alpha$-flat, for any $\alpha \geq 0$. \\

%$L^1$-norm and the $L^2$-norm doesn't implies $L^1$-flatness.\\

%We remind that
The Rudin-Shapiro polynomials were introduced independently by H. Shapiro \cite{Shapiro} and W. Rudin \cite{Rudin}. In the beginning, H. Shapiro introduced  them in his 1951's Master thesis \cite[p.39]{Shapiro}. Precisely, he produced them in the study of two extremal problems related to the well-known coefficient  problem of Landau \cite[p.139]{Duren}. Later, by appealing essentially to the same methods, W. Rudin produced those polynomials in \cite{Rudin} with the consent of H. Shapiro. Therein, the author answered positively a question posed to him by Salem. Since then, the Rudin-Shapiro polynomials are the powerful tool  for building examples or counterexamples in many settings. Among the most known examples is the one in which the Rudin-Shapiro polynomials are used to establish the sharpness in Bernstein's theorem \cite[Vol I, p.240]{Zygmund} which says that the set of Lipschitz functions $\textrm{Lip}(\alpha)$ is contained in the space of functions with an absolutely convergent Fourier series provided that  $\alpha>1/2$.\\

The Rudin-Shapiro polynomials are also connected to the spectral theory of dynamical systems. This discovery was made by T. Kaeme \cite{Ka} and M. Queffelec \cite{Q}. Indeed, they proved that the Rudin-Shapiro polynomials arise in the calculus of the spectral type of a certain substitution called nowadays the Rudin-Shapiro substitution. Their investigation yields that the spectrum of the Rudin-Shapiro substitution has a Lebesgue component of multiplicity two. Before, J. Mathew and M. G. Nadkarni constructed in \cite{MN} a transformation with a Lebesgue component of multiplicity two. Of course, this gives a positive answer to the weak Rokhlin problem. Nevertheless, it seems that Rokhlin's problem on finding an ergodic measure preserving
transformation on a finite measure space whose spectrum is Lebesgue type with finite multiplicity is still unsolved \cite[p.219]{Rokhlin}. This problem is related to the well-known Banach's problem whether a transformation measure-preserving
 may have a simple Lebesgue spectrum \cite[p.76]{Ulam}. For the Banach's problem, the transformation may act on the infinite measure space and it is supposed to be ergodic and non-dissipative.\\

  The  connection between Banach's problem and some extremal problem for analytic trigonometric polynomials was established by J. Bourgain in \cite{Bourgain}. Therein, he proved that the singularity of the spectrum of the rank one maps is related to the $L^1$-flatness of the analytic trigonometric polynomials with coefficients in $\{0,1\}$. Later, M. Guenais \cite{Guenais} and Downarowicz-Lacroix \cite{Down} proved that the weak Rokhlin's problem is connected to the $L^1$-flatness in the class of analytic trigonometric polynomials with coefficients in $\{+1,-1\}$. Precisely, M. Guenais proved that there exist a dynamical system with simple Lebesgue component if and only if there exist a sequence of analytic trigonometric polynomials $P_n, n=1,2,\cdots$ whose absolute values $\mid P_n\mid, n=1,2,\cdots$  converge to 1 in $L^1$, and such that for each $n$, the coefficients of $P_n$ are $\pm 1$. Downarowicz and Y. Lacroix established that the weak Rokhlin's problem has a positive answer provided that there exist the so-called Barker sequences of arbitrary high length.\\

Recently, E. H. el Abdalaoui and M. Nadkarni proved that the weak Rokhlin's problem has an affirmative answer
in the class of nonsingular transformations of a Lebesgue probability space \cite{AbN1}. Therein, the authors produced a sequence of ultraflat polynomials $P_n, n=1,2,\cdots$ with real coefficients, where for each $n$ the coefficients of $P_n, n=1,2,\cdots$ are not equal in absolute value, but which can be used to construct non-singular ergodic maps with the desired spectral property. We also mention that Guenais has constructed a concrete ergodic measure preserving group action whose unitary group admits a Haar component of multiplicity one in its spectrum. Guenais's construction is done by appealing to the ultraflat Fekete polynomials defined on some countable discrete torsion Abelian group \cite{elabdal-lem}.\\

Very recently, the author in \cite{Ab} settled the Banach Scottish problem by proving that there exists a rank one map acting on an infinite measure space with simple Lebesgue spectrum. Therein, the author established that there exist $L^1$-flat polynomials from the class $B$ or the so-called Newman polynomials. The construction of such polynomials is based on Singer's construction and on the nice combinatoric properties of Sidon sets and Singer sets. The proof is accomplished as a consequence of the sharper results from $H^p$ theory and the Carleson interpolation theory. Consequently, the author obtained a positive answer to the problems attributed to Littlewood and Mahler. For more details,  we refer to \cite{Ab}.\\
%\cite{Abd1} and \cite{Ab}.\\

Here, in connection with the Littlewood problem on the flatness of the polynomials with coefficients $\pm 1$, we shall prove that the Rudin-Shapiro polynomials are not $L^\alpha$-flat for $\alpha \geq 0$. This result supports the conjecture mentioned by D. J. Newman in \cite{Newman} which says that all the analytic trigonometric polynomials $P$ with coefficients $\pm 1$ satisfy
\begin{eqnarray*}
\big\|P\big\|_1 <c \big\|P\big\|_2,
\end{eqnarray*}
for some positive constant $c<1$.\\

According to M. Guenais's result \cite{Guenais} and Downarowicz-Lacroix contribution \cite{Down}, the previous conjecture can be translated in the language of dynamical systems and the language of coding and information theory respectively as follows:
\begin{Conj}[On the Morse cocyle]
Morse cocycles with simple Lebesgue component do not exist.
\end{Conj}
\begin{Conj}[On the Barker sequences]
Barker sequences of arbitrary length do not exist.
\end{Conj}
%yields also that it is most likely that the moment conjecture due to Safari  \cite{christof} has a positive answer in $L^\frac12$ .\\
We further give a short and a direct proof of the fact that the Fekete polynomials and the modified or shifted Fekete polynomials are not $L^\alpha$-flat, for any $\alpha \geq 0.$\\

In the forthcoming paper \cite{elabdal-little}, the author establishes the one-to-one correspondence between the Littlewood polynomials and the idempotent polynomials. Obviously, this gives a one-to-one correspondence between the $L^2$-normalized Littlewood polynomials and the Bourgain-Newman polynomials. As a consequence the author proves a criterion of $L^1$-flatness for those two class of polynomials. As an application of this criterion, the author shows that if the frequency of the positives coefficients in the Littlewood polynomials $(P_q)$ is not in the interval $[\frac14,\frac34]$, then $(P_q)$ are not $L^\alpha$- flat, $\alpha \geq 0$.\\

The paper is organized as follows. In section 2, we review some basic tools and we state our main result. In section 3, we prove it. Finally, in the appendix, we present a simple proof of the fact that Fekete polynomials and shifted Fekete polynomials are not $L^\alpha$-flat, for any $\alpha \geq 0.$

\section{Basic definitions and tools}
Let $S^1$ denote the circle group and $dz$ the normalized Lebesgue measure on $S^1$. The Rudin-Shapiro polynomials are defined inductively as follows
\[
\left\{
  \begin{array}{ll}
    P_0=Q_0=1, & \hbox{} \\
    P_{n+1}(z)=P_n(z)+z^{2^n}Q_n(z), & \hbox{ and} \\
    Q_{n+1}(z)=P_n(z)-z^{2^n}Q_n(z), & \hbox{for $n\geq0$ and $z \in S^1$.}
  \end{array}
\right.
\]
It can be easily seen by induction that $P_n$ and $Q_n$ are analytic trigonometric polynomials of degree $2^n-1$ with coefficients $\pm 1$. Let $\{r_k\}_{k=0}^{2^n-1}$ and $\{r'_k\}_{k=0}^{2^n-1}$ be respectively the sequence
of coefficients of the polynomials $P_n$ and $Q_n$. We thus have
$$r'_k=\left\{
         \begin{array}{ll}
           r_k, & \hbox{if $k <2^{n-1}$;} \\
           -r_k, & \hbox{if not.}
         \end{array}
       \right.
$$
The sequence $r_k$, $k=0,\cdots,$ is nowadays called the Golay-Rudin-Shapiro sequence \footnote{This terminology is due to Brillhart \& Morton \cite{Brillart2}.}. Such a sequence verifies
\begin{eqnarray}\label{RSprop}
r_0&=&1,\nonumber\\
r_n&=&\left\{
      \begin{array}{ll}
        r_{n/2}, & \hbox{if $n$ is even;} \\
        (-1)^{[n/2]}r_{[n/2]}, & \hbox{if $n$ is odd,}
      \end{array}
    \right.
\end{eqnarray}
as customary, $[x]$ denotes the integer part of $x$. It has been shown by Brillart and Carlitz \cite{Brillart} that if
 $n=\sum_{j=0}^{N_n}\epsilon_j2^j$ is a dyadic representation of $n$, then
$$r_n=(-1)^{s(n)},$$
where $s(n)=\sum_{j=0}^{N_n-1}\epsilon_j \epsilon_{j+1}$. Therefore, one can define $r_n$ in a combinatoric manner by
considering $s(n)$ as the occurrence number of “11” in the dyadic representation of $n$. For that it suffices to see that
$(-1)^{s(n)}$ satisfies the Golay-Rudin-Shapiro properties \eqref{RSprop}.\\

The relationship between the Golay-Rudin-Shapiro sequence and the substitution has been established by G. Christol, T. Kamae, M.
Mend\`es France and G. Rauzy in \cite{CKMR}. Indeed, therein, the authors proved that $r_n$ is a $2$-automatic sequence, obtained from a primitive substitution $\xi$ of length two on a four-letters alphabet $\{0,1,2,3\}$ given by
\begin{eqnarray*}
\xi(0)&=&02\\
\xi(1)&=&32\\
\xi(2)&=&01\\
\xi(3)&=&31,
\end{eqnarray*}
with the fixed point $u=\xi^{\infty}(0)$ which generated the Rudin-Shapiro dynamical system. We further have $r_n=\pi(u_n)$, where $\pi~~:~~\{0,1,2,3\}\longrightarrow\{+1,-1\}$ projects $0,2$ onto $1$ and $0,3$ on $-1$. This can be seen by defining the following words on the alphabet $\{+1,-1\}$
$$
A_n=r_0 \cdots r_{2^n-1},
{\textrm{~~and~~}}  B_n= r_{2^n}\cdots r_{2^{n+1}-1}.$$
We define the Rudin-Shapiro conjugate transform on the word $\ds A \in \bigcup_{n \in \N}\big\{+1,-1\big\}^n$ by
$$\overline{A}(i)=\left\{
    \begin{array}{ll}
      +1, & \hbox{if $A(i)=-1$;} \\
      -1, & \hbox{if not.}
    \end{array}
  \right.
$$
We thus have
$$
A_{n+1}=A_nB_n,\\{\textrm{~~and~~}} B_{n+1}=A_n\overline{B_n}.$$

Subsequently, we construct the fixed point by coding $A$ as $0$, $B$ as $1$, the conjugate of $A$ as $2$ and the conjugate of
$B$ as $3$. For the construction of the stage $n+1$, the rule is given by $\xi$. \\

We end this section by recalling that a sequence $P_n(z), n=1,2,\cdots$ of analytic trigonometric polynomials of $L^2(S^1,dz)$
norm 1 is said to be $L^1$-flat if the sequence $| P_n(z)|, n=1,2,\cdots$ converges in $L^1$-norm to the constant function $1$ as $n\rightarrow \infty$. More generally, for $\alpha>0$ or $\alpha=+\infty$, we say that $(P_n)$ is $L^\alpha$-flat if $| P_n(z)|, n=1,2,\cdots$ converges in $L^\alpha$-norm to the constant function $1$, . For $\alpha=0$, we say that $(P_n)$ is
$L^\alpha$-flat, if the Mahler measures $M(P_n)$ converge to 1. We remind that the Mahler measure of a function $f \in L^1(S^1,dz)$ is defined by
$$ M(f)=\|f\|_0=\lim_{\beta \longrightarrow 0}\|f\|_{\beta}=\exp\Big(\int_{S^1} \log(|f(t)|) dt\Big).$$

In the following, we generalize the characterization of $L^1$-flatness obtained in \cite{Ab} to the case of $L^\alpha$-flatness for $0<\alpha \leq 2$.
\begin{Prop}[Characterization of $L^\alpha$-flatness, $0 <\alpha \leq 2.$]
Let $(P_n)$ be a sequence of $L^2$-normalized analytic trigonometric polynomials and $0<\alpha \leq 2.$ Then, the following assertions are equivalent
\begin{enumerate}[(i)]
  \item $(P_n)$ are $L^{\alpha/2}$-flat.
  \item $\ds \big\|P_n\big\|_{\alpha/2} \tend{n}{+\infty}1.$
  \item $\ds \big\|\big|P_n\big|^{\alpha}-1\big\|_{1}\tend{n}{+\infty}0.$
\end{enumerate}
\end{Prop}
\begin{proof}We start by proving that $(ii)$ and $(iii)$ are equivalent. Assume that $(ii)$ holds. Then, by Cauchy-Schwarz inequality, we have
\begin{eqnarray}\label{Key1}
\int \big|\big|P_n(z)\big|^{\alpha}-1\big| dz&=& \int \big|P_n(z)\big|^{\alpha/2}-1\big| \big(\big|P_n(z)\big|^{\alpha/2}+1\big) dz \nonumber \\
&\leq& \big\|\big|P_n\big|^{\alpha/2}-1\big\|_2.\big\|\big|P_n\big|^{\alpha/2}+1\big\|_2.
\end{eqnarray}
But, since $0<\alpha \leq 2$, we have
\begin{eqnarray*}
\int \big|\big|P_n(z)\big|^{\alpha/2}+1\big|^2 dz &=& \int \big|P_n(z)\big|^\alpha dz +2 \int \big|P_n(z)\big|^{\alpha/2} dz+1 \\
&\leq& \big\|P_n\big\|_2^\alpha+ 2\big\|P_n\big\|_2^{\alpha/2}+1=4.
\end{eqnarray*}
Therefore, we can rewrite \eqref{Key1} as follows
$$\int \big|\big|P_n(z)\big|^{\alpha}-1\big| dz \leq 2 \big\|\big|P_n\big|^{\alpha/2}-1\big\|_2.$$
We further have
$$\int \big|\big|P_n(z)\big|^{\alpha/2}-1\big|^2 dz=\int \big|P_n(z)\big|^\alpha dz -2 \int \big|P_n(z)\big|^{\alpha/2} dz+1,
$$
and
$$\lim_{n }\int \big|P_n(z)\big|^\alpha dz=1,$$
since
$$\big\|P_n\big\|_{\alpha/2}\leq \big\|P_n\big\|_{\alpha} \leq \big\|P_n\big\|_{2}=1.$$
Whence
$$\int \big|\big|P_n(z)\big|^{\alpha/2}-1\big|^2 dz \tend{n}{+\infty}0,$$
and we conclude that $(iii)$ holds.
In the opposite direction, notice that we have
$$\int \big|\big|P_n(z)\big|^{\alpha}-1\big| dz \geq \int \big|\big|P_n(z)\big|^{\alpha/2}-1\big| dz,$$
since
$$\big(\big|P_n(z)\big|^{\alpha/2}+1\big) \geq 1.$$
Whence
$$\int \big|\big|P_n(z)\big|^{\alpha/2}-1\big| dz \tend{n}{+\infty}0.$$
This gives
$$\Big|\big\|P_n\big\|_{\alpha/2}^{\alpha/2}-1\Big|\tend{n}{+\infty} 0,$$
we thus get
$$\big\|P_n\big\|_{\alpha/2}\tend{n}{+\infty} 1.$$
We proceed now to prove that $(i)$ implies $(ii)$. By the triangle inequality, it is obvious that $(i)$ implies $(ii)$.
For $(ii)$ implies $(i)$, observe that $\big\|P_n\big\|_{\alpha/2}$ converges to $1$ as $n \longrightarrow +\infty$ implies that  $\big\|P_n\big\|_{1}$ converges to $1$ as $n \longrightarrow +\infty$, since $\alpha/2 \leq 1$. Moreover, $\big\|P_n\big\|_{1}$ converges to $1$ as $n \longrightarrow +\infty$ is equivalent to
$\big\||P_n|-1\big\|_1$ converges to $0$ as $n \longrightarrow +\infty$ (for the proof see \cite{Ab}.). We thus get
$$\big\||P_n|-1\big\|_{\frac{\alpha}2} \tend{n}{+\infty}0,$$
completing the proof of the proposition.
\end{proof}
 The flatness problem has nowadays a long history, and it turns out that the flat polynomials can be used to design
signals for use in digital transmission systems. Consequently, the increasing interest in the flatness problem is fueled by the need for pseudo-random sequences suitable for application in fields such as transmission and encryption.
 For more details, we refer the reader to \cite{Ab}, \cite{LaCour} and \cite{Jedwab}. Here, we  shall prove the following\\
\begin{Th}[{\bf Main result }]\label{main} The $L^2$-normalized sequence of Rudin-Shapiro polynomials is not $L^\alpha$-flat, for any
$\alpha \geq 0$.
\end{Th}
\section{Proof of the main result}
Applying the parallelogram  law, we can write
\begin{eqnarray}\label{para}
\big|P_{n+1}(z)\big|^2+\big|Q_{n+1}(z)\big|^2=2^{n+2},
\end{eqnarray}
for any $n=0,1,2,\cdots$. From this, it is easy to see that we have
\begin{eqnarray}\label{RSKey1}
\big\|P_{n}\big\|_{\infty} \leq \sqrt{2} \big\|P_{n}\big\|_{2}.
\end{eqnarray}
We further have that $L^\alpha$-norm and $L^2$-norm are equivalent. Indeed, suppose $0<\alpha < 2$ and write
\begin{eqnarray*}
\big\|P_{n}\big\|_{2}^2 &=&\int_{S^1} \big|P_{n}(z)\big|^{2-\alpha} \big|P_{n}(z)\big|^{\alpha} dz\\
&\leq&  \big\|P_{n}(z)\big\|_{\infty}^{2-\alpha}  \big\|P_{n}\big\|_{\alpha}^{\alpha}.
\end{eqnarray*}
 By appealing to \eqref{RSKey1}, we get
$$\big\|P_{n}\big\|_{2}^2 \leq 2^{\frac{2-\alpha}2}  \big\|P_{n}\big\|_{2}^{2-\alpha} \big\|P_{n}\big\|_{\alpha}^{\alpha},$$
whence
$$\big\|P_{n}\big\|_{2} \leq 2^{\frac{2-\alpha}{2\alpha}}\big\|P_{n}\big\|_{\alpha}.$$
If $\alpha>2$, then by the same reasoning, we have
$$\big\|P_{n}\big\|_{\alpha} \leq 2^{\frac{\alpha-2}{2\alpha}}\big\|P_{n}\big\|_{2}.$$
Summarizing we have proved the following lemma.
\begin{lem}For any $\alpha>0$ and for $n \in \N^*,$ we have
$$c_\alpha \big\|P_{n}\big\|_{\alpha} \leq \big\|P_{n}\big\|_{2} \leq C_\alpha  \big\|P_{n}\big\|_{\alpha},$$
and
$$c_\alpha \big\|Q_{n}\big\|_{\alpha} \leq \big\|Q_{n}\big\|_{2} \leq C_\alpha \big\|Q_{n}\big\|_{\alpha},$$
where
$$c_\alpha=\left\{
             \begin{array}{ll}
              1 , & \hbox{if $0<\alpha<2$,} \\
              2^{\frac{2-\alpha}{2\alpha}} , & \hbox{if $\alpha>2$,}
             \end{array}
           \right.
{\textrm {~~and~~~}}~~~~~~~~~~
C_\alpha=\left\{
             \begin{array}{ll}
               2^{\frac{2-\alpha}{2\alpha}} & \hbox{if $0<\alpha<2$;} \\
               1 , & \hbox{if $\alpha>2$.}
             \end{array}
           \right.
$$
\end{lem}
We further notice that \eqref{RSKey1} can be strengthened as follows
\begin{lem}[Rudin-Shapiro's lemma ] For any $N \in \N^*$, we have
$$\Big\|\sum_{n=0}^{N}r_n z^n\Big\|_{\infty} \leq 5 \Big\|\sum_{n=0}^{N}r_n z^n\Big\|_{2}.$$
\end{lem}
The previous lemma is due independently to Rudin \cite{Rudin} and Shapiro \cite{Shapiro}.\\

Let us put
$$R_N(z)=\sum_{n=0}^{N}r_n z^n.$$
Applying the same reasoning as before, it is a simple matter to establish the following proposition.
\begin{Prop} For any $\alpha> 0$ there exist  positive real numbers $C_\alpha$ and $D_\alpha$ such that
$$C_\alpha \Big\|\frac1{\sqrt{N}}R_N\Big\|_2 \leq \Big\|\frac1{\sqrt{N}}R_N\Big\|_\alpha \leq D_\alpha \Big\|\frac1{\sqrt{N}}R_N\Big\|_2,$$
for all $N \in \N^*$.
\end{Prop}
This proposition can be strengthened as follows.
\begin{Prop}\label{Katz} For any $\alpha > 0$ there exist positive real numbers $C_\alpha$ and $D_\alpha$ such that
$$C_\alpha \Big\|\frac1{\sqrt{M}}\sum_{n=N+1}^{N+M}r_n z^n\Big\|_2 \leq \Big\|\frac1{\sqrt{M}}\sum_{n=N+1}^{N+M}r_n z^n\Big\|_\alpha \leq D_\alpha \Big\|\frac1{\sqrt{M}}\sum_{n=N+1}^{N+M}r_n z^n\Big\|_2,$$
for all $N,M \in \N^*$.
\end{Prop}
The proof of Proposition \ref{Katz} is a simple consequence of the following lemma  \cite{Kat}, \cite[p.134]{KahaneS}.
\begin{lem}For all positive integers $N, M$, we have
$$ \Big|\sum_{n=N+1}^{N+M}r_n z^n\Big| \leq \sqrt{M}.$$
\end{lem}
We further remind that the spectral measure of the Rudin-Shapiro sequence is the Lebesgue measure. This result is due to Kamae  \cite{Ka}. An alternative proof of it was given by M. Queffelec  in \cite{Q}. Indeed, she established that the sequence of probability measures $(\nu_N=\big|\frac1{\sqrt{N}}R_N(z)\big|^2dz)$ converges weakly to the  Lebesgue measure.
%From this it can be deduced that the spectrum of the Rudin-Shapiro has a Lebesgue component of multiplicity two. Besides, J. Mathew and M. G. Nadkarni constructed in \cite{MN} a transformation with a Lebesgue component of multiplicity two.
Moreover, we have the following estimation of the correlation of the Rudin-Shapiro due to Mauduit-S\'{a}rk\"{o}zy \cite{MauduitS}.
\begin{Prop}For any positive integer $k$ and $N$, we have
\begin{eqnarray}\label{CF-estim1}
 &&\big| \widehat{\nu_N}(k) \big| <\frac{2k}{N}+\frac{4k}{N}\log_2\Big(\frac{2N}{k}\Big),
\end{eqnarray}
and
\begin{eqnarray}\label{CF-estim2}
 &&\max_{l \leq N} \big| \widehat{\nu_N}(l) \big| \geq \frac16.
\end{eqnarray}
\end{Prop}
%\noindent{}Let us put
%$$m_{i,j}=\int z^{i-j}d\nu_N-c_ic_j,$$
%where
%$\ds c_k=\widehat{\nu_N}(k),$ for any $k$. Let us also define the following the quantity
%$$C(N)=\sum_{\overset{-(N-1) \leq i,j \leq N-1}{i \neq j}}\big|m_{i,j}\big|.$$
%The following criterium of flatness is proved in \cite{AbN}
%\begin{Th}
%\end{Th}

We are now able to prove our main result. Put
$$X_n=\frac{P_n}{\big\|P_n\big\|_2}=\frac{P_n}{2^{n/2}},$$
and
$$Y_n=\frac{Q_n}{\big\|Q_n\big\|_2}=\frac{Q_n}{2^{n/2}}.$$
We can thus rewrite \eqref{para} as
\begin{eqnarray}\label{para2}
|X_n|^2+|Y_n|^2=2,~~~~~~~~~~~\forall n \in \N.
\end{eqnarray}
Furthermore, it can be seen by induction that
\begin{eqnarray}\label{Billart}
Y_n=(-1)^n z^{2^n-1}X_n(-1/z),
\end{eqnarray}
for any $n \geq 0$ and for any $z \in S^1$. This result is due to Brillhart \& Carlitz \cite{Brillart}.  Applying \eqref{Billart}, we get
$$\big\|X_n\big\|_\alpha=\big\|Y_n\big\|_\alpha,$$
for any $\alpha \geq 0$. Formula \eqref{Billart} allows us also to prove the following theorem due to Littlewood \cite{Littlewood} and Newman-Bynes \cite{NewmanB}. For the convenience of the reader, we include its proof.
\begin{Th}[The $L^4$-norm theorem ]\label{fourth} $\ds
\big\|X_n\big\|_4 \tend{n}{+\infty}\sqrt[4]{\frac43}.$
\end{Th}
\begin{proof}Let $n \in \N^*$ and put
$$x_n=\big\|X_n\big\|_4^4+\big\|Y_n\big\|_4^4=\int_{S^1} \big|X_n\big|^4+\big|Y_n\big|^4 dz=2\big\|X_n\big\|_4^4,$$
and notice that we have the following identity
\begin{eqnarray*}
|P_n|^2|Q_n|^2&=&\Big|\big(P_{n-1}+z^{2^{n-1}}Q_{n-1}\big) \big(P_{n-1}-z^{2^{n-1}}Q_{n-1}\big)\Big|^2\\
&=& \big|P_{n-1}\big|^4+ \big|Q_{n-1}\big|^4-2\Re\Big(z^{2^{n}}P_{n-1}\overline{Q_{n-1}}\Big).
\end{eqnarray*}
We thus get
$$\int_{S^1} |X_n|^2 |Y_n|^2 dz=\frac14 \int_{S^1} \Big(\big|X_{n-1}\big|^4+\big|Y_{n-1}\big|^4\Big) dz=\frac14 x_{n-1},$$
by integration. Since $z^{2^{n}}P_{n-1}\overline{Q_{n-1}}$ is analytic by \eqref{Billart}. Now, squaring equation \eqref{para2} gives
$$x_n+\frac12 x_{n-1}=4.$$
A straightforward application of the well-known formula on the sequence generated by the first-order linear recurrence relation gives
$$x_n=-\frac{2}{3}\Big(-\frac{1}2\Big)^n+\frac{8}3.$$
Whence
$$ \big\|X_n\big\|_4^4= -\frac{1}{3}\Big(-\frac{1}2\Big)^n+\frac{4}3.$$
Letting $n\longrightarrow +\infty$, we conclude that $\ds \big\|X_n\big\|_4^4 \longrightarrow \frac{4}3.$ The proof of
the theorem is complete.
\end{proof}
Applying Mauduit-S\'{a}rk\"{o}zy estimation \eqref{CF-estim2}, we can estimate the $L^4$-norm of the Rudin-Shapiro polynomials as
follows.
\begin{Prop}\label{4-norm}$\ds \liminf_{N\longrightarrow +\infty}\Big\|\frac{R_N}{\sqrt{N}}\Big\|_4 \geq \sqrt[4]{\frac{19}{18}}.$
\end{Prop}
%For the proof we need the following Lemma:
%\begin{lem}Let $n$ be a positive integer. Then, for any $k<2^{n-1}$,
%$r_{2^n+k}=r_k.$
%Furthermore, if $N=2^{n_k}+2^{n_{k-1}}+\cdots+2^{n_1}$ with $n_k>n_{k-1}>\cdots>n_1.$ Then
%$$R_N(z)=P_{n_k}(z)+z^{2^{n_k}}R_{N-2^{k-1}}(z).$$
%\end{lem}
\begin{proof}A straightforward computation gives
$$\Big|\frac{R_N(z)}{\sqrt{N}}\Big|^4=1+\sum_{|l| \leq N-1}\widehat{\nu_N}(l)z^l,~~~~~~~~~\forall z \in S^1.$$
Whence
$$\Big\|\frac{R_N}{\sqrt{N}}\Big\|_4^4=1+2\sum_{l=1}^{N-1}\big|\widehat{\nu_N}(l)\big|^2.$$
This combined with \eqref{CF-estim2} yields
$$\Big\|\frac{R_N}{\sqrt{N}}\Big\|_4^4 \geq 1+\frac2{36}=\frac{19}{18},$$
which proves the proposition.
\end{proof}
In the language of digital communications engineering, we have proved that the merit factor of the Golay-Rudin-Shapiro sequence is bounded by $18$. We remind that the merit factor of a sequence $u=\{u_n\}_{n}$ is given by
$$F_N(u)=\frac{\big|c_0\big|^2}{2E(N)}, ~~~~~~~~~\textrm{for~~all~~} N \in \N^*,$$
where $E(N)$ is the energy defined by
$$E(N)=\sum_{k=1}^{N}|c_k(N)|^2,$$
and for each $k \in \{0,\cdots, N\}$, $c_k(N)$ is the correlation of order $k$ given by
$$c_k(N)=\sum_{i=0}^{N-k-1}u_i \overline{u_{i+k}}.$$
Following complex analysis language, this notion can be defined as follows:\\

Let $N$ be a positive integer and put
$$U_N(z)=\sum_{j=0}^{N}u_j z^j, \forall z \in S^1.$$
Then
$$\big|U_N(z)\big|^2=c_0+\sum_{1 \leq |k| \leq N-1}c_k(N)z^k,$$
and by taking the $L^2$-norm of the polynomial $\big|U_N(z)\big|^2$, it follows that
$$\Big\|U_N\Big\|_4^4=|c_0|^2+2 \sum_{1 \leq k \leq N-1}\big|c_k(N)\big|^2.$$
Whence, the merit factor of $u$ can be defined by the following identity
$$\frac{\Big\|U_N\Big\|_4^4}{|c_0|^2}=1+\frac1{F_N(u)}.$$
We notice that $|c_0|^2=\big\|U_N\big\|_2^2.$  We thus define the merit factor of $U$ by
$$F_N(u)=\frac{\Big\|U_N\Big\|_2^2}{\Big\|U_N\Big\|_4^4-\Big\|U_N\Big\|_2^2}.$$
For a nice account on the merit factor problem, we refer the reader to \cite{Jedwab}.\\

We proceed to prove by contradiction our main result for $\alpha=1$. Suppose that the sequence $(R_N)$ is $L^1$-flat. Then, we can assert that
$$\int_{S^1} \big||X_n|-1\big|dz \tend{n}{+\infty}0.$$
%Whence
%$$\int \big||Y_n|^2-1\big|dz \tend{n}{+\infty}0,$$
%Thanks to \eqref{para2} which can be rewritten in the following form
%$$ |X_n|^2-1=-\Big(|Y_n|^2-1\Big).$$
%Applying the triangle inequality, we get
%$$\Big\||X_n|^2-|Y_n|^2\Big\| \tend{n}{+\infty}0.$$
%But, a straightforward computation gives
It follows that we can extract a subsequence $(X_{n'})$ such that $|X_{n'}|$ converges a.e. to $1$. But $(|X_n|^4)_{n \geq 0}$ is uniformly bounded, thanks to \eqref{para2}. Therefore, by Lebesgue's dominated convergence theorem, we can assert that
$\big\|X_n\big\|_4$ converges to $1$, which contradicts the $L^4$-norm theorem (Theorem \ref{fourth}). We thus conclude that
$(X_n)$ is not $L^1$-flat.\\

\noindent{}We further claim that we have
$$\limsup_{N \longrightarrow +\infty} \big\|R_N \big\|_1<1.$$
Indeed, by contradiction, suppose that we have
$$\limsup_{N \longrightarrow +\infty} {\big\|R_N \big\|_1}=1.$$
Then, along a subsequence $(N_j)$, we can assert that
$\ds {\big\|R_{N_j} \big\|_1} \longrightarrow 1$ as $j\longrightarrow+\infty$. It follows that
$\Big\|\big|R_{N_j} \big|^2-1\Big\|_1 \longrightarrow 0$ as $j\longrightarrow+\infty$. Indeed, by Cauchy-Schwarz inequality, we have
\begin{eqnarray*}
\Big\|\big|R_{N_j} \big|^2-1\Big\|_1 &\leq& \Big\|\big|R_{N_j} \big|-1\Big\|_2 \Big\|\big|R_{N_j} \big|^2+1\Big\|_2\\
&\leq& 2 \Big\|\big|R_{N_j} \big|-1\Big\|_2.
\end{eqnarray*}
The last inequality is due to the triangle inequality combined with $||R_N||_2 =1$. We further have
$$\Big\|\big|R_{N_j} \big|-1\Big\|_2^2 =2\Big(1-\Big\|R_{N_j}\Big\|_1\Big).$$
This gives that there exists a subsequences which we denote again by $N_j$, such that $(|R_{N_j}|)$ converges almost everywhere to $1$. Therefore, by Lebesgue dominated convergence theorem,
$(|R_{N_j}|)$ converges to $1$ in the $L^4$-norm. We thus get $\big\|R_{N_j}\big\|_4 \longrightarrow 1$ as $j \longrightarrow +\infty$,   which contradicts Proposition \ref{4-norm}. By the same arguments, we can conclude that $(R_N)$ is not $L^\alpha$-flat for any $\alpha>0$, and since $\|f\|_0 \leq \|f\|_\alpha,$ for any $\alpha>0$. It follows that $(R_N)$ is not $L^\alpha$-flat, for any $\alpha \geq 0$. This finishes the proof of our main result.
\QEDA
\begin{rem}
In the connection with the merit factor problem for the class of \linebreak Littlewood sequences $\mathcal{L}=\ds \bigcup_{N=1}^{+\infty}\big\{+1,-1\big\}^N$, Newman and Byrnes in \cite{NewmanB} conjectured the following
\begin{Conj}[\textbf{of Newman and Byrnes}]
$$\ds \liminf_{N \longrightarrow +\infty}\Big(\min_{u \in {\big\{+1,-1\big\}^N} }\Big\|U_N(z)\Big\|_4\Big) \geq \sqrt[4]{\frac65}.$$
\end{Conj}
Following the authors, this conjecture seems to be supported by the extensive numerical evidence employing the Bose-Einstein statistics methodology of statistical mechanic. Notice that this conjecture solves the merit factor problem which says that
the minimum of the $L^4$-norm of the $L^2$-normalized analytic polynomials with $\pm 1$ coefficients is great than $1+\delta$ for some $\delta>0$. We remind that the analytic polynomials with $\pm 1$ coefficients are nowdays called Littlewood polynomials. Of course, this conjecture implies that the $L^4$-uniformly integrable sequence of \linebreak Littlewood polynomials cannot be $L^1$-flat. We further notice that Newman-Byrnes's conjecture implies that Erd\"{o}s conjecture holds. Let us recall that Erd\"{o}s conjecture for the class of Littlewood polynomials says that for any $L^2$-normalized Littlewood polynomials $P$, we have $\ds \big\|P\big\|_{\infty} \geq (1+C)$ for some constant $C>0$ \cite{Erdos}.
\end{rem}
\begin{thank}
The author wishes to express his thanks to Fran\c cois \linebreak Parreau, Jean-Paul Thouvenot, Benjamin Weiss, Michael Lin, Bill Veech, Eli \linebreak Glasner, Jon Aaronson, Mahendra Nadkarni, Doureid Hamdan and Bernard
Host for many stimulating conversations on the subject.
\end{thank}
\appendix
\section{On the Fekete polynomials.}
Let $p$ be a prime number. The Fekete polynomials are defined by
$$Q_p(z)=\sum_{k=1}^{p-1}\left(\frac{k}{p}\right)z^k,$$
where $\ds \left(\frac{k}{p}\right)$ is the Legendre symbol. We recall that the Legendre symbol is given by
$$
\left(\frac{k}{p}\right)=\left\{
    \begin{array}{ll}
      0, & \hbox{if $k=0$;} \\
      1, & \hbox{if a square modulo~$p$;} \\
      -1, & \hbox{if not.}
    \end{array}
  \right.
$$
The Fekete polynomials are intimately linked to the Gauss sum given by
$$G_2(\chi)=\sum_{n}\chi(n) \omega_p^n=Q_p(\omega_p),$$
where $\chi(n)=\ds \left(\frac{k}{p}\right)$ and $\omega_p$ is the $p$-root of unity given by
$\omega_p=e^{2 \pi i 1/p}.$ In his 1811's paper, Gauss proved the following
$$Q_p(\omega_p^k)=\epsilon_p \sqrt{p}\left(\frac{k}{p}\right),~~~~~~~k=0,\cdots,p-1,$$
where $\epsilon_p$ is defined by
$$\epsilon_p=\left\{
               \begin{array}{ll}
                 1, & \hbox{if $p \equiv 1$ (mod 4);} \\
                 i, & \hbox{if $p \equiv 3$ (mod 4).}
               \end{array}
             \right.
$$
For the proof of the Gauss formula, we refer to \cite[p.10]{BerndtEW}. We thus get
\begin{eqnarray}\label{Gauss}
\big|Q_p(\omega_p^k)\big|=\sqrt{p},
\end{eqnarray}
for any $k=0,\cdots,p-1$.
It follows that the Fekete polynomials may possibly be a candidate to solve the Littlewood's well-known
problem whether there exists a sequence of analytic trigonometric polynomials $P_n$ with coefficients $\pm 1$  such that
$$c.\big\|P_n\big\|_2 \leq \big|P_n(z)\big| \leq C.\big\|P_n\big\|_2,$$
for all $z \in S^1$ and for two absolute positive constants $c, C$. Unfortunately, \linebreak Montgomery in \cite{Montgomery} proved that there is an absolute constant $c$ such that
$$\frac2{\pi}\sqrt{p}\log(\log(p)) < \big\|Q_p\big\|_{\infty}\egdef\sup_{z \in S^1}\big|Q_p(z)\big| \leq c \sqrt{p} \log(p).$$
In \cite{Erdelyi-Lub}, Erd\'elyi and Lubinsky observed that if $p \equiv 1$ (mod 4) then the Fekete polynomials are not $L^1$-flat  by applying the following Littlewood's criterion.
\begin{Th}[Littlewood's criterion \cite{Littlewood2}] Let $\ds f_n(t)=\sum_{j=0}^{n}a_m \cos(m t+\phi_m)$ and assume that we have
$$\sum_{m=1}^{n}a_m^2 \leq \frac{k}{n^2} \sum_{m=1}^{n}m^2a_m^2,$$
for some absolute constant $k$. Then, for any $\alpha>0$ there exists a constant $A(k,\alpha)$ such that
$$
  \begin{array}{ll}
    \|f_n\|_\alpha \leq \big(1-A(k,\alpha)\big)\|f_n\|_2, & \hbox{ if $\alpha<2$;} \\
    \|f_n\|_\alpha \geq \big(1+A(k,\alpha)\big)\|f_n\|_2, & \hbox{ if $\alpha>2$.}
  \end{array}
$$
\end{Th}
\noindent{}In order to apply the criterion of Littlewood we observe that if $p\equiv 1$  $($mod $4)$ then $Q_p$ is self-reciprocal, that is,
$$z^{p-1}Q_p\Big(\frac1{z}\Big)=Q_p(z).$$
Hence
$$Q_p(t)=e^{2it}\sum_{k=0}^{\frac{p-3}2}a_k \cos((2k+1)t),~~~~~~~a_k=\pm 2.
$$
\QEDA\\

Here, we shall prove the following.
\begin{Th}\label{mainF} The Fekete polynomials are not $L^\alpha$-flat for any $\alpha \geq 0$.
\end{Th}
For that we need the following theorem due H{\o}holdt and Jensen  \cite{Hoholdt-Jensen}.
\begin{Th}The $L^4$-norm of the $L^2$-normalized Fekete polynomials verify
$$ \frac{\|Q_p\|_4}{p}\tend{p}{+\infty}\sqrt[4]{\frac53}.$$
\end{Th}
The exact formula of $\|Q_p\|_4$ in term of the number of class of $\mathbb{Q}(\sqrt{-p})$ was obtained by Browein and Choi \cite{Browein-Choi}. Very recently, G\"{u}nther and Schmidt in \cite{Gunther-Schmidt} established a formula for the limit of the $L^{2q}$-norm where $q$ is a positive integer.\\

We further needs some  classical and fundamental ingredient from the modern probabilistic literature. This ingredient seems to be useful in some problem related to the almost sure convergence.\\

\noindent{} Of course, it is obvious that almost everywhere convergence does not in general imply
convergence in $L^p(X)$. Nevertheless, it is well known that the condition of domination ensures such convergence (Lebesgue's Dominated Convergence Theorem) but in the absence of domination the following Vitali's convergence theorem allows us to obtain the convergence in $L^p(X)$ provided that the sequence is uniformly integrable.
\begin{Th}[Vitali's convergence theorem ] \label{Vitali}Let $(X,\mathcal{B},\mu)$ be a probability space, $p$ a positive number and $\{f_n\}$ a sequence in $L^p(X)$ which converges in probability to $f$. Then, the following are equivalent:
\begin{enumerate}[(i)]
\item $(|f_n|^p)_{n \geq 0}$ is uniformly integrable;
\item $\ds
\Big|\Big|f_n-f\Big|\Big|_p \tend{n}{+\infty}0.$
\item $\ds \int_X |f_n|^p d\mu \tend{n}{+\infty} \int_X |f|^p d\mu.$
\end{enumerate}
\end{Th}
\noindent{}For the proof of Theorem \ref{Vitali} we refer the reader to \cite[pp.101]{Chung}.\\

\noindent{}We remind that the condition
\begin{eqnarray*}\label{uniform}
\sup_{ n \in \N}\Big(\int_X\big|f_n\big|^{1+\varepsilon}d\mu\Big) < +\infty,
\end{eqnarray*}
for some $\varepsilon$ positive, implies that $\{f_n\}$ are uniformly integrable.\\

The third ingredient we need is the Marcinkiewicz-Zygmund inequality \cite[ Vol. II, p. 30]{Zygmund}. The Marcinkiewicz-Zygmund inequality asserts that for any $\alpha \in ]1,+\infty[$ there are two positive constants $A(\alpha), B(\alpha)$ such that
\begin{eqnarray}\label{MZ}
\frac{A_{\alpha}}{n}\sum_{j=0}^{n-1}\big|P(e^{2\pi i\frac{j}{n}})\big|^{\alpha}
\leq \int_{\T}\Big|P(z)\Big|^{\alpha} dz \leq \frac{B_{\alpha}}{n}\sum_{j=0}^{n-1}\big|P(e^{2\pi i\frac{j}{n}})\big|^{\alpha},
\end{eqnarray}
for any polynomial $P$ of degree at most $n-1$.\\

We are now ready to prove Theorem \ref{mainF}.
\begin{proof}[\textbf{Proof of Theorem \ref{mainF}}]We start by noting that for any $\alpha \geq 0$, there is a constant $c_\alpha$ such that,
$$\|Q_p\|_\alpha \leq c_\alpha,$$
for any prime number $p$. This is an easy application of the Marcinkiewicz-Zygmund inequality \eqref{MZ} combined with Gauss formula \eqref{Gauss}. Assume that for some $\alpha>0$, $Q_p$ is $L^\alpha$-flat, Then
$$\Big\| \frac{\big|Q_p(z)\big|}{\sqrt{p}}-1\Big\|_\alpha \tend{p}{+\infty}0.$$
This gives that there is a subsequence $(p_n)$ for which the sequence $\Big(\frac{|Q_{p}|}{\sqrt{p}}\Big)$ converges almost everywhere to $1$. But $\Big(\frac{|Q_{p_n}|^4}{p_n}\Big)$ is uniformly integrable, hence
$$\Big\|\frac{Q_{p_n}}{\sqrt{p_n}}\Big\|_4 \tend{n}{+\infty}1$$
by Vitali's convergence theorem \ref{Vitali}. This contradicts Theorem \ref{4-norm}. We thus get that for any $\alpha \geq 0,$
 the Fekete polynomials are not $L^\alpha$-flat as in the proof of Theorem \ref{main}.
\end{proof}
\begin{rem}Of course, Theorem \ref{mainF} is valid for the modified Fekete polynomials and the shifted Fekete polynomials defined respectively by
$$F_p(z)=1+Q_p(z), \textrm{~~for~~any~~} z \in S^1,$$
and
$$F_p^t(z)=\sum_{k=0}^{q-1}\left(\frac{k+t}{p}\right)z^k.$$
\end{rem}
\begin{que}Our work suggests the following natural question.
Let $p$ be a prime number and  $q = p^2 + p + 1$. Let $S \subset \Z/q\Z$ be a Singer set and put
$T = S + S$. By the nice combinatorial properties of Singer sets we know that if $x = a+b$
then this representation is essentially unique: there is one representation if
$a = b$ and there are two if a $a \neq b$ , namely $x = a + b = b + a$. Whence
$$\big|T\big| =\frac12\big(p^2+3p+2),$$
where, as is customary, $\big|T\big|$ denotes  the number of elements in the set $T$.  Define the sequence $(\epsilon_{j}^{(q)})_{j=0}^{q-1}$ by
$$\epsilon_{j}^{(q)}=\left\{
    \begin{array}{ll}
      1, & \hbox{if $j \in T$;} \\
      -1, & \hbox{if not.}
    \end{array}
  \right.
$$
For any $z \in S^1$, put
$$P_q(z)=\frac1{\sqrt{q}}\sum_{j=0}^{q-1}\epsilon_{j}^{(q)} z^j.$$
Can one prove or disprove that $(P_q)$ is $L^\alpha$-flat, for $\alpha>0$.
\end{que}

  \end{document}